\documentclass[10pt]{amsart}
\usepackage{amssymb,amsthm,amsmath,amsfonts}
\usepackage{hyperref}
\usepackage{pstricks}
\usepackage{graphicx}
\usepackage{pst-node}

\def\R{{\mathbb R}}
\def\Rr{{\mathcal R}}
\def\C{{\mathbb C}}

\def\Z{{\mathbb Z}}

\def\<{\langle}
\def\>{\rangle}

\def\E{{\mathbb E}}
\def\U{{\mathbb U}}
\def\V{{\mathbb V}}

\def\X{\mathbb X}

\def\Y{\mathbb Y}

\newcommand{\be}{\begin{equation}}
\newcommand{\ee}{\end{equation}}
      \newtheorem{theorem}{Theorem}[section]

       \newtheorem{lemma}[theorem]{Lemma}

\title{A constant regression characterization of a Marchenko-Pastur law}
\author[K. Szpojankowski]{Kamil Szpojankowski}
\address[K. Szpojankowski]{Wydzia\l{} Matematyki i Nauk Informacyjnych\\
Politechnika Warszawska\\
ul. Koszykowa 75\\
00-662 Warszawa, Poland}
\email{k.szpojankowski@mini.pw.edu.pl}


\begin{document}
\begin{abstract}
Lukacs type characterization of Marchenko--Pastur distribution in free probability is studied here. We prove that for free $\mathbb{X}$ and $\mathbb{Y}$ when conditional moments of order $1$ and $-1$ of $(\mathbb{X+Y})^{-1/2}\mathbb{X}(\mathbb{X+Y})^{-1/2}$ given $\mathbb{X+Y}$ are constant then $\mathbb{X}$ and $\mathbb{Y}$ have Marchenko--Pastur distribution.
\end{abstract}
\maketitle
\section{Introduction}

Since the paper \cite{VoiculescuAdd} free probability theory has been developed in many various directions. It turns out that many classical results for independent random variables such as for example Central Limit Theorem have their free analogues. One of the deepest known relations between classical and free probability are so called Bercovici -- Pata bijections which give bijection between infinitely divisible distributions in free and classical convolution. 

In this paper we are interested in characterization problems in free probability. This seems to be another field which gives interesting connection between classical and free probability. Our result is a new example of known, but not completely well understood phenomena of analogies between characterizations in classical and free probability. A basic example of such analogy is the Bernstein's theorem which characterizes a Gaussian distribution by independence of $X+Y$ and $X-Y$ for independent $X$ and $Y$. In \cite{NicaChar} it is proved that similar result holds for the Wigner semicircle law when independence is replaced by freeness assumption. 

The main result of this paper is closely related to the Lukacs theorem which provides a characterization of a Gamma distribution by independence of $V=X+Y$ and $U=X/(X+Y)$ for independent $X$ and $Y$ (see \cite{Lukacs}). It is known that the assumption of independence of $U$ and $V$ can be replaced by a weaker assumption of constancy of regressions $\E\,\left(U|V\right)$ and $\E\,\left(U^2|V\right)$ (see \cite{LahaLukacs}). In \cite{WesolGamma} it is proved that constancy of regressions $\E\,\left(U|V\right)$ and $\E\,\left(U^{-1}|V\right)$ also characterizes a Gamma distribution.

The Lukacs property was also studied in a context of free probability in \cite{BoBr2006}, where Laha-Lukacs regression of free Meixner family is studied (see also \cite{EjsmontLL,EjsmontTD}). Theorem 3.1 from \cite{BoBr2006} contains as a special case a free analogue of Lukacs regressions in the case of constancy of regressions of $U$ and $U^2$ given by $\V$. It turns out that such conditions characterize the Marchenko--Pastur (free Poisson) distribution. The proof that the Marchenko--Pastur distributed $\X$ and $\Y$ have the property that $\V=\X+\Y$ and $\U=(\X+\Y)^{-1/2}\X(\X+\Y)^{-1/2}$ are free can be found in \cite{SzpLukTheo}. In this paper we prove a free analogue of the Lukacs regressions in the case of regressions $\U$ and $\U^{-1}$ given by $\V$. The proof of the main result relies mainly on the technique developed in our previous papers \cite{SzpojanWesol,SzpDLRNeg}.

The paper is organized as follows: in the Section 2 we briefly introduce basic notions of free probability and known facts which are needed to prove the main result. In Section 3 we state and prove the main result of the paper.

\section{Preliminaries}

In this section we give a collection of facts which we need in this paper. For more detailed introduction we refer to our previous papers \cite{SzpojanWesol,SzpDLRNeg}. A comprehensive introduction to free probability can be found in \cite{VoiDykNica} or \cite{NicaSpeicherLect}.

By a non-commutative probability space we understand a pair $\left(\mathcal{A},\varphi\right)$ where $\mathcal{A}$ is a unital algebra and $\varphi$ is a faithful, normal, tracial state.

For self-adjoint and free $\X$, $\Y$ with distributions $\mu$ and $\nu$, respectively, and $\X$ positive, that is the support of $\mu$ is a subset of $(0,\infty)$, free multiplicative convolution of $\mu$ and $\nu$ is defined as the distribution of $\sqrt{\X}\,\Y\sqrt{\X}$ and denoted by $\mu\boxtimes\nu$. Due to the tracial property of $\varphi$ the moments of $\Y\,\X$, $\X\,\Y$ and $\sqrt{\X}\,\Y\sqrt{\X}$ match.

Let $\chi=\{B_1,B_2,\ldots\}$ be a  partition of the set of numbers $\{1,\ldots,k\}$. A partition $\chi$ is a crossing partition if there exist distinct blocks $B_r,\,B_s\in\chi$ and numbers $i_1,i_2\in B_r$, $j_1,j_2\in B_s$ such that $i_1<j_1<i_2<j_2$. Otherwise $\chi$ is called a non-crossing partition. The set of all non-crossing partitions of $\{1,\ldots,k\}$ is denoted by $NC(k)$.

For any $k=1,2,\ldots$, (joint) cumulants of order $k$ of non-commutative random variables $\X_1,\ldots,\X_n$ are defined recursively as $k$-linear maps $\mathcal{R}_k:\C\,\langle x_i,\,i=1,\ldots,k\rangle\to\C$ through equations
$$
\varphi(\Y_1,\ldots,\Y_m)=\sum_{\chi\in NC(m)}\,\prod_{B\in\chi}\,\mathcal{R}_{|B|}(\Y_i,\,i\in B)
$$
holding for any $\Y_i\in\{\X_1,\ldots,\X_n\}$, $i=1,\ldots,m$, and any $m=1,2,\ldots$,
with $|B|$ denoting the number of elements in the block $B$.

Freeness can be characterized in terms of behavior of cumulants in the following way: Consider unital subalgebras $(\mathcal{A}_i)_{i\in I}$ of an algebra $\mathcal{A}$ in a non-commutative probability space $(\mathcal{A},\,\varphi)$. Subalgebras $(\mathcal{A}_i)_{i\in I}$ are freely independent iff for any $n=2,3,\ldots$ and for any $\X_j\in\mathcal{A}_{i(j)}$ with $i(j)\in I$, $j=1,\ldots,n$ any $n$-cumulant
$$
\mathcal{R}_n(\X_1,\ldots,\X_n)=0
$$
if there exists a pair $k,l\in\{1,\ldots,n\}$ such that $i(k)\ne i(l)$.

In sequel we will use the following formula from \cite{BozLeinSpeich} which connects cumulants and moments for non-commutative random variables
\be\label{BLS}
\varphi(\X_1\ldots\X_n)=\sum_{k=1}^n\,\sum_{1<i_2<\ldots<i_k\le n}\,\mathcal{R}_k(\X_1,\X_{i_2},\ldots,\X_{i_k})\,\prod_{j=1}^k\,\varphi(\X_{i_j+1}\ldots\X_{i_{j+1}-1})
\ee
with $i_1=1$ and $i_{k+1}=n+1$ (empty products are equal 1).

For free random variables $\X$ and $\Y$ having distributions $\mu$ and $\nu$, respectively, the distribution of $\X+\Y$, denoted by $\mu\boxplus\nu$, is called free convolution of $\mu$ and $\nu$.

Non-commutative conditional expectation is well defined in a so called $W^*$ probability spaces, i.e. non-commutative probability spaces where algebra $\mathcal{A}$ is a von Neumann algebra. Non-commutative conditional expectation has many properties analogous to those of classical conditional expectation. For more details one can consult e.g. \cite{Takesaki}. Here we state two of them we need in the sequel.
\begin{lemma}\label{conexp} Consider a $W^*$-probability space $(\mathcal{A},\varphi)$.
	\begin{itemize}
		\item If $\X\in\mathcal{A}$ and $\Y\in\mathcal{B}$, where $\mathcal{B}$ is a von Neumann subalgebra of $\mathcal{A}$, then
		\be\label{ce1}
		\varphi(\X\,\Y)=\varphi(\varphi(\X|\mathcal{B})\,\Y).
		\ee
		\item If $\X,\,\Z\in\mathcal{A}$ are free then
		\be\label{ce2}
		\varphi(\X|\Z)=\varphi(\X)\,\mathbb{I}.
		\ee
	\end{itemize}
\end{lemma}

Now we introduce basic analytical tools used to deal with non-commutative random variables and their distributions.

For a non-commutative random variable $\X$ its $r$-transform is defined as
\begin{align*}
r_{\X}(z)=\sum_{n=0}^{\infty}\,\mathcal{R}_{n+1}(\X)\,z^n.
\end{align*}
In \cite{VoiculescuAdd} it is proved that $r$-transform of a random variable with compact support is analytic in a neighbourhood of zero.  From properties of cumulants it is immediate that for $\X$ and $\Y$ which are freely independent
\be\label{freeconv}
r_{\X+\Y}=r_{\X}+r_{\Y}.
\ee
If $\X$ has the distribution $\mu$, then often we will write $r_{\mu}$ instead $r_{\X}$.\\
The Cauchy transform of a probability measure $\nu$ is defined as
$$
G_{\nu}(z)=\int_{\R}\,\frac{\nu(dx)}{z-x},\qquad \Im(z)>0.
$$
Cauchy transforms and $r$-transforms are related by
\be\label{Crr}
G_{\nu}\left(r_{\nu}(z)+\frac{1}{z}\right)=z.
\ee
Finally we introduce moment generating function $M_{\X}$ of a random variable $\X$ by
\begin{align*}
M_{\X}(z)=\sum_{n=1}^{\infty}\,\varphi(\X^n)\,z^n.
\end{align*}
It is easy to see that
\begin{align}\label{MSr}
M_{\X}(z)=\frac{1}{z}G_{\X}\left(\frac{1}{z}\right).
\end{align}
We will need the following lemma proved in \cite{SzpDLRNeg}.
\begin{lemma}
	\label{lem_cum}
	Let $\V$ be compactly supported, invertible non-commutative random variable. Define $C_n=\mathcal{R}_{n}\left(\V^{-1},\V,\ldots,\V\right)$, and $C(z)=\sum_{i=1}^{\infty}C_iz^{i-1}$. Then for $z$ in neighbourhood of $0$ we have
	\be 
	\label{lem_1}
	C(z)=\frac{z+C_1}{1+zr(z)},
	\ee
	where $r(z)$ is $R$-transform of $\V$. In particular,
	\be 
	\label{lem_2}
	C_2=1-C_1\mathcal{R}_1(\V),\ \ C_n=-\sum_{i=1}^{n-1}C_i\mathcal{R}_{n-i}(\V),\, n\geq 2
	\ee
\end{lemma}

\section{Main Result}

In this section we study a regressive characterization of the Marchenko-Pastur distribution which is a free counterpart of the characterization of the gamma distribution proved in \cite{WesolGamma}.

\begin{theorem}\label{DualLukacs1}
Let  $(\mathcal{A},\,\varphi)$ be $W^*$-probability space, let $\X,\,\Y$ be non-commutative random variables in $(\mathcal{A},\,\varphi)$. Assume that $\X$ and $\Y$ are free, $\X$ is strictly positive, $\Y$ is positive and there exist real numbers $c$ i $d$, such that
\begin{align}
\label{regL1}
\varphi\left(\left.\X\right|\X+\Y\right)=c\,\left(\X+\Y\right)
\end{align}
and
\begin{align}
\label{regL2}
\varphi\left(\left.\X^{-1}\right|\X+\Y\right)=d\,\left(\X+\Y\right)^{-1}.
\end{align}
Then $\X$ and $\Y$ have free Poisson distributions $\nu(c \lambda,\alpha)$ and $\nu((1-c)\lambda,\alpha)$ respectively, where $\lambda=\frac{d-1}{cd-1}$ and $\alpha=\frac{cd-1}{C_1(1-c)}$ for some $C_1>0$.
\end{theorem}
\begin{proof}
Multiplying \eqref{regL1} and \eqref{regL2} by $(\X+\Y)^n$ and applying the state to both sides of the equations, we obtain for $n\geq 0$
\begin{align}
\label{regL12}
\varphi\left(\X\left(\X+\Y\right)^n\right)&=c\,\varphi\left((\X+\Y)^{n+1}\right),\\
\label{regL22}\varphi\left(\X^{-1}\left(\X+\Y\right)^n\right)&=d\,\varphi\left((\X+\Y)^{n-1}\right).
\end{align}
Let us define three sequences $(\alpha_n)_{n\ge -1}$, $(\beta_n)_{n\ge 0}$ and $(\delta_n)_{n\ge 0}$ as,
\begin{align*}
\alpha_n=\varphi\left(\left(\X+\Y\right)^n\right),\qquad
\beta_n=\varphi\left(\X\left(\X+\Y\right)^n\right)\qquad\mbox{and\qquad}\delta_n=\varphi\left(\X^{-1}\left(\X+\Y\right)^n\right).
\end{align*}
We can rewrite \eqref{regL12} and \eqref{regL22}
as
\begin{align}
\label{regL13}\beta_n=c\,\alpha_{n+1},\\
\label{regL23}\delta_n=d\,\alpha_{n-1}.
\end{align}
Multiplying both sides of the above equations by $z^n$ and summing over $n=0,1,\ldots$ we get
\begin{align}
\label{regL14}B(z)&=c\,\frac{1}{z}\left(A(z)-1\right),\\
\label{regL24}D(z)&=d\,z\left(A(z)+\frac{\alpha_{-1}}{z}\right),
\end{align}
where
\begin{align*}
A(z)=\sum_{n=0}^\infty\alpha_nz^n,\qquad
B(z)=\sum_{n=0}^\infty\beta_nz^n,\qquad
D(z)=\sum_{n=0}^\infty\delta_nz^n.
\end{align*}
Using the formula \eqref{BLS} and freeness of $\X$ and $\Y$, for sequence $\alpha_n$ we get
\begin{align*}
\beta_n=&\Rr_1\alpha_n\\
+&\Rr_2\left(\alpha_{n-1}+\alpha_{n-2}\alpha_1+\ldots+\alpha_{n-1}\right)\\
+&\ldots+\Rr_{n+1}.
\end{align*}
where $\Rr_n=\Rr_n\left(\X\right).$\\
For $n\geq 0$ this gives us
\begin{align*}
\beta_n=\sum_{k=1}^{n+1}\Rr_k\sum_{i_1+\ldots+i_k=n+1-k}\alpha_{i_1}\ldots\alpha_{i_k}.
\end{align*}
Using the above equations we get
\begin{align*}
B(z)&=\sum_{n=0}^{\infty}z^n\beta_n=\sum_{n=0}^{\infty}z^n\sum_{k=1}^{n+1}\Rr_k\sum_{i_1+\ldots+i_k=n+1-k}\alpha_{i_1}\ldots\alpha_{i_k}\\
&=\sum_{k=1}^{\infty}z^{k-1}\Rr_k\sum_{n=k-1}^{\infty}\sum_{i_1+\ldots+i_k=n+1-k}\alpha_{i_1}z^{i_1}\ldots\alpha_{i_k}z^{i_k}\\
&=\sum_{k=1}^{\infty}z^{k-1}\Rr_k\sum_{m=0}^{\infty}\sum_{i_1+\ldots+i_k=m}\alpha_{i_1}z^{i_1}\ldots\alpha_{i_k}z^{i_k}.\\
\end{align*}
Which means that
\begin{align}
\label{RLuk_fA}
B(z)=A(z)r_{\X}(zA(z)),
\end{align}
where $r_{\X}(z)=\sum_{n=0}^{\infty}\Rr_{n+1} z^n$. Note that $r_{\X}$ is the $r$-transform of $\X$.

Next we proceed similarly with the sequence $\delta_n$
\begin{align*}
\delta_n=&C_1\alpha_n\\
+&C_2\left(\alpha_{n-1}+\alpha_{n-2}\alpha_1+\ldots+\alpha_{n-1}\right)\\
+&\ldots+C_{n+1},
\end{align*}
where $C_n=\Rr_n\left(\X^{-1},\underbrace{\X,\ldots,\X}_{n-1}\right),$ dla $n\geq 0$.\\
Which means that for $n\geq 0$ we have
\begin{align*}
\delta_n=\sum_{k=1}^{n+1}C_k\sum_{i_1+\ldots+i_k=n+1-k}\alpha_{i_1}\ldots\alpha_{i_k}
\end{align*}
The above equation gives us
\begin{align*}
D(z)&=\sum_{n=0}^{\infty}z^n\delta_n=\sum_{n=0}^{\infty}z^n\sum_{k=1}^{n+1}C_k\sum_{i_1+\ldots+i_k=n+1-k}\alpha_{i_1}\ldots\alpha_{i_k}\\
&=\sum_{k=1}^{\infty}z^{k-1}C_k\sum_{n=k-1}^{\infty}\sum_{i_1+\ldots+i_k=n+1-k}\alpha_{i_1}z^{i_1}\ldots\alpha_{i_k}z^{i_k}\\
&=\sum_{k=1}^{\infty}z^{k-1}C_k\sum_{m=0}^{\infty}\sum_{i_1+\ldots+i_k=m}\alpha_{i_1}z^{i_1}\ldots\alpha_{i_k}z^{i_k}.\\
\end{align*}
This implies that
\begin{align*}
D(z)=A(z)C(zA(z)),
\end{align*}
where $C(z)=\sum_{n=0}^{\infty}C_{n+1} z^n$.
Using lemma \ref{lem_cum} we get
\begin{align}
\label{RLuk_fD}
D(z)=A(z)\frac{zA(z)+C_1}{1+zA(z)r_{\X}(zA(z))}.
\end{align}
Using the equations \eqref{RLuk_fA} and \eqref{RLuk_fD},  we can rewrite \eqref{regL14} and \eqref{regL24} as
\begin{align*}
A(z)r_{\X}(zA(z))&=c\,\frac{1}{z}\left(A(z)-1\right),\\
A(z)\frac{zA(z)+C_1}{1+zA(z)r_{\X}(zA(z))}&=d\,z\left(A(z)+\frac{\alpha_{-1}}{z}{}\right).
\end{align*}
Let us define an auxiliary function $h(z)=zA(z)r_{\X}(zA(z))$, then we can rewrite the above equations as
\begin{align}
\label{regL15}
h(z)&=c\,(A(z)-1),\\
\label{regL25}
A(z)\frac{zA(z)+C_1}{1+h(z)}&=d\,z\left(A(z)+\frac{\alpha_{-1}}{z}{}\right).
\end{align}
Since $h(0)=0$, then in some neighbourhood of zero we can multiply \eqref{regL25} by $1+h$. Taking into account that the equation \eqref{regL2} implies $C_1=d\,\alpha_{-1}$, we get
\begin{align*}
zA^2(z)+A(z)C_1-zA(z)d\,(1+h(z))-C_1(1+h(z))=0.
\end{align*}
In the above equation we can replace one function $A$ in first and second term by $\frac{h+c}{c}$ which follows from  \eqref{regL15}. After simple transformations we obtain
\begin{align}
\label{wzor_fh}
\frac{h(z)}{zA(z)}=\frac{c(d-1)}{C_1(1-c)-zA(z)(cd-1)}.
\end{align}
Recall that  $h(z)=zA(z)r_{\X}\left(zA(z)\right)$. Since $r$ is analytic in a neighbourhood of 0 and $\lim_{z\to 0}zA(z)=0$ we get
\begin{align}
\label{rtr_X}
r_{\X}(z)=\frac{c(d-1)}{C_1(1-c)-z(cd-1)}.
\end{align}
From the equation \eqref{regL1}, and the assumption that $\X$ and $\Y$ are positive we get $$c=\varphi(\X)/\varphi\left(\X+\Y\right)\in(0,1).$$
Similarly freeness of $\X$ and $\Y$ gives us
$$d=\varphi\left(\X^{-1}(\X+\Y)\right)=1+\varphi\left(\X^{-1}\right)\varphi(\Y)>1.$$
The Cauchy-Schwartz inequality implies $cd>1$.\\
This means that $\X$ has free Poisson distribution with parameters $\lambda=\frac{c(d-1)}{cd-1}$ and $\alpha=\frac{cd-1}{C_1(1-c)}$.

Next we shall determine the distribution of $\Y$.\\
Substituting in equation \eqref{regL15} $h$ from \eqref{wzor_fh}, we get
\begin{align*}
A^2(z) z (cd-1)+A(z)(zd(1-c)-C_1(1-c))+C_1(1-c)=0.
\end{align*}
Since $A$ is the moment transform of $\X+\Y$, we can use the connection between moment and Cauchy transforms, and after substituting $z:=1/z$ we obtain
\begin{align*}
G_{\X+\Y}^2(z) z (cd-1)+G_{\X+\Y}(z)d(1-c)-G_{\X+\Y}(z) z C_1(1-c)+C_1(1-c)=0.
\end{align*}
Now using the equation \eqref{Crr} we get the $r$-transform of $\X+\Y$
\begin{align*}
r_{\X+\Y}(z)=\frac{d-1}{C_1(1-c)-(cd-1)z}.
\end{align*}
Using \eqref{freeconv} we get
\begin{align*}
r_{\Y}(z)=\frac{(1-c)(d-1)}{C_1(1-c)-(cd-1)z}
\end{align*}
Which means that $\Y$ has the free Poisson distribution with parameters $\lambda=\frac{(1-c)(d-1)}{cd-1}$ and $\alpha=\frac{cd-1}{C_1(1-c)}$.
\end{proof}

\subsection*{Acknowledgement} The author thanks J. Weso\l{}owski for many helpful comments and discussions. This research was partially supported by NCN
grant 2012/05/B/ST1/00554.
\bibliographystyle{plain}
\bibliography{Bibl}
\end{document}